\def\timestamp{%
Time-stamp: <many-subalgebras.tex: maandag 06-04-2026 at 12:39:08 (cest)>}
\def\stripname Time-stamp: <#1: #2 #3 at #4 #5>{#3/#4 (#1)}
\edef\filedate{\expandafter\stripname\timestamp}
\documentclass[a4paper]{amsart}
\newcommand\dom{\operatorname{dom}}
\newcommand\St{\operatorname{St}}
\newcommand\Pwmodfin{\mathcal{P}(\omega)/\mathit{fin}}
\newcommand\ceeseq[2][\beta]{\langle{#2}_{#1}:#1\in\cee\rangle}
\newcommand\omegaseq[1]{\langle{#1}_n:n\in\omega\rangle}
\newcommand\Nseq[1]{\langle{#1}_n:n\in\N\rangle}
\newcommand\orpr[2]{\langle{#1},{#2}\rangle}

\newcommand\calC{\mathcal{C}}
\newcommand\calF{\mathcal{F}}

\newcommand\cee{\mathfrak{c}}

\DeclareMathSymbol\A          \mathord{AMSb}{`A}
\DeclareMathSymbol\N          \mathord{AMSb}{`N}
\DeclareMathSymbol\Q          \mathord{AMSb}{`Q}
\DeclareMathSymbol\notsubseteq\mathrel{AMSb}{"2A}
\DeclareMathSymbol\le    \mathrel{AMSa}{"36}
\DeclareMathSymbol\ge    \mathrel{AMSa}{"3E}

\newtheorem{theorem}{Theorem}[section]
\newtheorem{lemma}[theorem]{Lemma}
\newtheorem{proposition}[theorem]{Proposition}

\theoremstyle{remark}
\newtheorem*{remark}{Remark}

\usepackage{amsrefs}
\newcommand\Zbl{\ifhmode\unskip\spacefactor3000 \space\fi zbMATH\,}
\begin{document}

\title{Many subalgebras of $\Pwmodfin$}

\author[K. P. Hart]{Klaas Pieter Hart}
\address{Faculty EEMCS\\TU Delft\\
         Postbus 5031\\2600~GA {} Delft\\the Netherlands}
\email{k.p.hart@tudelft.nl}
\urladdr{https://fa.ewi.tudelft.nl/\~{}hart}

\begin{abstract}
In answer to a question on Mathoverflow we show that the Boolean
algebra $\mathcal{P}(\omega)/\mathit{fin}$ contains a family 
$\{\mathcal{B}_X:X\subseteq\mathfrak{c}\}$ of subalgebras
with the property that $X\subseteq Y$ implies $\mathcal{B}_Y$ is a subalgebra
of $\mathcal{B}_X$ and if $X\not\subseteq Y$ then $\mathcal{B}_Y$
is not embeddable into~$\mathcal{B}_X$.
The proof proceeds by Stone duality and the construction of a suitable
family of separable zero-dimensional compact spaces.
\end{abstract}


\date{\filedate}

\maketitle

\section*{Introduction}

The purpose of this note is to give a more leisurely presentation, complete
with definitions and references, of an answer to a question 
on \texttt{MathOverflow}~\cite{overflow442111}:
\begin{quote}
Is there a strictly decreasing chain of subalgebras
 of the Boolean algebra $\Pwmodfin$?
\end{quote}
The answer to the question as stated is an obvious ``yes'', but the poser
of the question asked for a sequence $\omegaseq B$ of subalgebras
such that $B_{n+1}\subseteq B_n$ \emph{and} $B_n$~is \emph{not} embeddable 
into~$B_{n+1}$, for all~$n$.

We shall show that the family of subalgebras of $\Pwmodfin$~is rich enough
to contain such a sequence; in fact, there is a family 
$\{B_X:X\subseteq\cee\}$ of subalgebras with the property that
for all subsets $X$ and~$Y$ of~$\cee$ we have: 
if $X\subseteq Y$ then $B_Y\subseteq B_X$ \emph{and} if $X\notsubseteq Y$ then
$B_Y$~is \emph{not} embeddable into~$B_X$.
This more than answers the question and shows that one can even have 
a decreasing chain of length~$\cee$ or a chain of order type that of the 
real line.

The construction of the family proceeds via Stone duality: rather
than constructing subalgebras of~$\Pwmodfin$ we construct a family
$\{K_X:X\subseteq\cee\}$ of separable compact zero-dimensional spaces
with the dual property that there is a  
set $\{h_{X,Y}:X\subseteq Y\subseteq\cee\}$ of continuous maps, where
$h_{X,Y}:K_X\to K_Y$ is a continuous surjection \emph{and}
if $X\notsubseteq Y$ then $K_Y$~is \emph{not} a continuous image of~$K_X$.
In addition all triangles in the set of maps will commute.

Then $K_\emptyset$ is a continuous image of~$\omega^*$, the Stone space 
of~$\Pwmodfin$, and hence so are all other spaces~$K_X$.
The maps $h_{\emptyset,X}$ embed the algebras of clopen sets of the~$K_X$
into~$\Pwmodfin$, the commutativity of the triangles in the family of 
continuous surjections yields the desired inclusions, and the nonexistence of 
further continuous surjections dualizes to the nonexistence of 
further embeddings.

\if00
\begin{remark}
It turned out, after this paper was written, that the original question 
had been answered 40~years before it was asked.
In~\cite{MR0735899} Murray Bell constructed a sequence~$\omegaseq{B}$ of 
subalgebras of~$\Pwmodfin$ with the property that for every $n\ge2$
the algebra~$B_n$ is $\sigma$-$n$-linked but not $\sigma$-$(n+1)$-linked.

At the end of the paper we will see how this sequence can be used to answer
the original question.
\end{remark}
\fi

\section{Preliminaries}

\subsection{Stone duality}

Stone's duality for Boolean algebras and compact zero-di\-men\-sio\-nal
spaces associates with every compact zero-dimensional space~$X$ its
Boolean algebra~$\calC_X$ of closed-and-open subsets and conversely
with every Boolean algebra~$B$ a compact zero-dimensional space~$\St(B)$,
its \emph{Stone space}.
The associations are each others inverses and they dualize various notions;
the most important for us is that an embedding $B\to C$ of Boolean algebras 
becomes a continuous surjection $\St(C)\to\St(B)$, and vice versa.  

The book~\cite{MR991565}*{Chapter~3} contains further information on 
Stone's duality for Boolean algebras and compact zero-dimensional spaces.

\subsection{Bernstein sets}

In our construction we shall use Bernstein sets in~$[0,1]$.
We say $A$~is a \emph{Bernstein set} in~$[0,1]$ if $A$ and its complement
both intersect every uncountable closed set in~$[0,1]$.
These are also called totally imperfect sets (\cite{zbMATH02640876})
because if a set is closed in~$[0,1]$ and contained in~$A$ then it must be 
countable.

For other topological material we refer to~\cite{MR1039321}.
The fact used here, that separable compact spaces are continuous images
of the remainder~$\omega^*$, can be proved using Theorem~3.5.13 
and Exercise~3.5.H of that book.

\section{The spaces}

The spaces are variations on Alexandroff's double-arrow space~$\A$,
called the two arrows space in~\cite{MR1039321}*{Exercise~3.10.C}.

The underlying set is
$D=\bigl([0,1]\times\{0,1\}\bigr)\setminus
   \bigl\{\orpr00,\orpr11\bigr\}$, 
ordered lexicographically and endowed with the order topology.
(We drop the points $\orpr00$ and $\orpr11$ because they
would be (the only) isolated points of $\A$.)

Pictorially we have taken the unit interval $[0,1]$ and split each point~$x$
of the open interval $(0,1)$ into two copies, $\orpr x0$ and $\orpr x1$.
The space~$\A$ is compact and separable, hence a continuous image
of~$\omega^*$.

The variations will be obtained by specifying a subset $X$ of $(0,1)$ and taking
$\A_X=\{\orpr xi: x\in X \to i=0\}$; that is, 
by splitting the points of $(0,1)\setminus X$ only.
Thus we can write $\A=\A_\emptyset$, and $[0,1]=\A_{(0,1)}$ for example.
In all our examples the complement of~$X$ will be dense in~$(0,1)$ and this
will ensure that $\A_X$~is zero-dimensional.

If $X\subseteq Y$ then there is a natural continuous surjection
$s:\A_X\to\A_Y$, given by
\begin{itemize}
\item $s(x,i)=\orpr xi$ if $x\notin Y$;
\item $s(x,i)=\orpr x0$ if $x\in Y\setminus X$; and
\item $s(x,0)=\orpr x0$ if $x\in X$. 
\end{itemize}

Our goal will be to create a family $\{S_X:X\subseteq\cee\}$ of subsets
of~$(0,1)$ such that with $K_X=\A_{S_X}$ for all~$X$ we get our family
$\{K_X:X\subseteq\cee\}$.

We shall construct a family $\{A_\alpha:\alpha\in\cee\}$ of subsets
of $(0,1)$ (all disjoint from $\Q$)
and put 
 $S_X=\Q\cup\bigcup_{\alpha\in X}A_\alpha$ for $X\subseteq\cee$.

Clearly then $X\subseteq Y$ implies $S_X\subseteq S_Y$ and hence
that $K_X$ maps onto~$K_Y$ by $h_{X,Y}:A_{S_X}\to\A_{S_Y}$ as described above.
It is readily seen that $h_{X,Z}=h_{Y,Z}\circ h_{X,Y}$, so all triangles in 
this family commute, as described in the introduction.

It remains to construct the sets $A_\alpha$ in such a way that whenever
$X\notsubseteq Y$ each of the $A_\alpha$ with $\alpha\in X\setminus Y$ 
will prohibit the existence of a continuous surjection from~$K_X$ onto~$K_Y$.

\medskip
To see how this may be accomplished note that since 
$A_\alpha\subseteq S_X$ the points of~$A_\alpha$ are not split in~$\A_{S_X}$.
In that case the subspace topology that $A_\alpha$~inherits from~$\A_{S_X}$
is the same as the subspace topology that it inherits from~$[0,1]$.

If $s:K_X\to K_Y$ is continuous then the composition $t\circ s$, where
$t:K_Y\to[0,1]$ is the map that sends $\orpr xi$ to~$x$, is continuous
as well and its restriction~$g$ to~$A_\alpha$ is also continuous.
We shall arrange matters in such a way that the only maps that can appear 
in this way will force the range of the map~$s$ to be countable.

\section{The sets $A_\alpha$}

We can obtain our sets $A_\alpha$ by a direct application of Theorem~2.0 
in~\cite{MR808722} but to keep this note reasonably self-contained
we shall repeat the construction for the special case that we need.  
The method goes back to~\cite{zbMATH03006477} and is 
occasionally referred to as 
``Sierpi\'nski's technique of killing homeomorphisms'' \cite{MR1173263},
but it can be used to eliminate other maps as well.

In our case we consider the set $\calF$ of all maps~$f$ that satisfy:
$\dom f$~is a co-countable subset of~$[0,1]$ and
$f:\dom f\to[0,1]$ is continuous.
For every $f\in\calF$ we let 
$S(f)=\{x\in\dom f:f(x)\neq x\}$ and 
$E(f)=\dom f\setminus S(f)$.
We choose a subset $C(f)$ of $\dom f$ such that the 
restriction $f:C(f)\to f\bigl[S(f)\bigr]$ is a bijection.

Before we continue we make some remarks that will be useful later.
For each $f\in\calF$ the domain is completely metrizable as it is a 
$G_\delta$-subset of~$[0,1]$.
As $S(f)$~is open in~$\dom f$, and $E(f)$~is closed, both sets are completely
metrizable as well.
Furthermore, the image $f\bigl[S(f)\bigr]$ is an analytic subset of~$[0,1]$.
By familiar results from Descriptive Set Theory it follows that each of these
sets either is countable or contains a topological copy of the Cantor set.
One can modify the construction outlined  
in~\cite{MR1039321}*{Problem~4.5.5} to prove these results.

This means that in each case we can check whether the set is countable 
by looking at its intersection with some Bernstein set.

The following proposition yields the family $\{A_\alpha:\alpha\in\cee\}$.

\begin{proposition}\label{prop.family}
There is a pairwise disjoint family 
$\{V\}\cup\{A_\alpha:\alpha\in\cee\}$ 
of Bernstein sets in~$(0,1)$ with the following properties.  
All are disjoint from $\Q$, and 
for every $f\in\calF$: if $f\bigl[S(f)\bigr]$, and hence $C(f)$, 
has cardinality~$\cee$
then for all~$\alpha$ the intersections $C(f)\cap A_\alpha$ 
and $f\bigl[C(f)\cap A_\alpha\bigr]\cap V$ both have cardinality~$\cee$.
\end{proposition}

\begin{proof}
Since $[0,1]$ has cardinality~$\cee$ it also has $\cee$~many co-countable
subsets.
Since each subset of~$[0,1]$ is separable every co-countable set has
$\cee$~many continuous functions to~$[0,1]$.
Hence we may enumerate the subfamily of~$\calF$ consisting of those~$f$
for which $f\bigl[S(f)\bigr]$~is uncountable as~$\ceeseq f$,
in such a way that every~$f$ occurs $\cee$~many times in the sequence. 
We take a similar enumeration $\ceeseq F$ of the family of uncountable
closed subsets of~$[0,1]$ (each set is listed $\cee$~times).

To facilitate the construction we replicate both enumerations $\cee$~times
and turn them into $\cee\times\cee$-matrices:
$\{f_{\alpha,\beta}:\orpr\alpha\beta\in\cee^2\}$ and
$\{F_{\alpha,\beta}:\orpr\alpha\beta\in\cee^2\}$,
where $f_{\alpha,\beta}=f_\beta$ and $F_{\alpha,\beta}=F_\beta$ for 
all $\alpha$ and $\beta$.
We also take a well-order~$\prec$ of~$\cee^2$ in order type~$\cee$.

By recursion on the well-order~$\prec$ we will choose points
$a_{\alpha,\beta}$,
$b_{\alpha,\beta}$,
$u_{\alpha,\beta}$, and 
$v_{\alpha,\beta}$,
as follows.

When the points have been found for $\orpr\gamma\delta\prec\orpr\alpha\beta$
collect them and the rational numbers in a set: 
$P=\Q\cup\bigcup_{\orpr\gamma\delta\prec\orpr\alpha\beta}
   \{a_{\gamma,\delta}, b_{\gamma,\delta}, u_{\gamma,\delta}, v_{\gamma,\delta}\}$.
Note that the cardinality of~$P$ is strictly smaller than~$\cee$.
Therefore we can find $a_{\alpha,\beta}\in C(f_{\alpha,\beta})\setminus P$ such that
$u_{\alpha,\beta}=f_{\alpha,\beta}(a_{\alpha,\beta})\notin P$;
and note that $u_{\alpha,\beta}\neq a_{\alpha,\beta}$.
Next take points $b_{\alpha,\beta}$ and $v_{\alpha,\beta}$ in 
$F_{\alpha,\beta}\setminus(P\cup\{a_{\alpha,\beta},v_{\alpha,\beta}\})$
such that $b_{\alpha,\beta}\neq v_{\alpha,\beta}$.

Now, by construction all points chosen in this way are distinct.
For every $\alpha\in\cee$ we let
$$
A_\alpha=\{a_{\alpha,\beta}:\beta\in\cee\}\cup\{b_{\alpha,\beta}:\beta\in\cee\}
$$
and we let
$$
V=\{u_{\alpha,\beta}:\orpr\alpha\beta\in\cee^2\}
     \cup\{v_{\alpha,\beta}:\orpr\alpha\beta\in\cee^2\}
$$
Because all points chosen are distinct the family
$\{V\}\cup\{A_\alpha:\alpha\in\cee\}$
is pairwise disjoint.

The sets are Bernstein sets because
$A_\alpha\cap F\supseteq\{b_{\alpha,\beta}: F=F_{\alpha,\beta}\}$
and 
$V\cap F\supseteq\{v_{\alpha,\beta}: F=F_{\alpha,\beta}\}$,
both intersections have cardinality~$\cee$.

Likewise, if $f\bigl[S(f)\bigr]$ has cardinality~$\cee$ then
$A_\alpha\cap C(f)\supseteq\{a_{\alpha,\beta}:f=f_{\alpha,\beta}\}$
and
$V\cap f[A_\alpha\cap C(f)]\supseteq\{u_{\alpha,\beta}:f=f_{\alpha,\beta}\}$;
again both sets have cardinality~$\cee$.
\end{proof}

Note that since $V$ is disjoint from $\Q$ and all sets $A_\alpha$ it is also
disjoint from all sets $S_X$.
This makes good on the promise that the sets used as input for the construction
have a dense complement: the Bernstein set $V$ is dense.

It now remains to show that the resulting family
$\{K_X:X\subseteq\cee\}$ of compact zero-dimensional spaces
has the desired properties.
We already know that $K_X$ maps onto~$K_Y$ if $X\subseteq Y$.
We prove the other implication in the next section.
There it will become clear what the function of the set~$V$ is,
and why $\Q$~is a subset of~$S_X$ for all~$X$.

\section{Non-existence of continuous surjections}

The following lemma implies that if $X$ and $Y$ are subsets of~$\cee$ 
such that $X\notsubseteq Y$ then there
is no continuous surjection from $K_X$ onto~$K_Y$. 

\begin{lemma}\label{lemma.range}
Let $X$ and $Y$ be subsets of $(0,1)$ such that $\Q\subseteq X$ and such that
there is an~$\alpha$ for which $A_\alpha\subseteq X$ 
and $Y\cap(A_\alpha\cup V)=\emptyset$.
Then every continuous map $s:\A_X\to\A_Y$ has a countable range.
\end{lemma}

\begin{proof}
Let us write $A$ for $A_\alpha$ and let $t:\A_Y\to[0,1]$ be the natural 
surjection.
Also, we identify $x$ and $\orpr x0$ when $x\in X$.

As observed before the topology on~$A$ in $\A_X$ is the same as its subspace
topology in~$[0,1]$.
Let $g$ be the restriction of $(t\circ s)$ to~$A$.

By one half of Lavrentieff's theorem 
(Theorem 4.3.20 in~\cite{MR1039321})
we can find a $G_\delta$-set~$G$ that contains~$A$ and a continuous 
map $f:G\to[0,1]$ that extends~$g$.
The complement, $C$, of~$G$ in~$[0,1]$ is a countable union of closed sets, 
each of which is countable because closed sets that are disjoint from
a Bernstein set are countable; and $A$~is a Bernstein set.
It follows that $f$ belongs to the family $\calF$.

\smallskip
The set $A$ is dense in $[0,1]$, and the maps $f$ and $t\circ s$ agree on~$A$.
This implies that $f$ determines much of the behaviour of~$s$ on~$G$,
in the following way.
\begin{itemize}
\item If $x\in G\cap X$ then $x$~is not split and $(t\circ s)(x)=f(x)$ and 
      this implies that $s(x)\in\bigl\{\orpr{f(x)}0, \orpr{f(x)}1\bigr\}$.
\item If $x\in G\setminus X$ then $x$~is split and the continuity of~$f$
      implies that $(t\circ s)(x,0)=(t\circ s)(x,1)=f(x)$ and so
  $\bigl\{s(x,0),s(x,1)\bigr\}\subseteq\bigl\{\orpr{f(x)}0,\orpr{f(x)}1\bigr\}$.
\end{itemize}
It follows that the range of~$s$ is contained in the union of
$f[G]\times\{0,1\}$ and the image of the countable set of points whose
first coordinates are in the countable set~$C$.

We finish the proof by showing that $f[G]$ is countable.

\smallskip
Let $x\in E(f)\cap A$, then $f(x)=x$ and so $s(x)=\orpr x0$ or $s(x)=\orpr x1$.
We divide $E(f)\cap A$ into two sets: $E_0=\{x:s(x)=\orpr x0\}$
and $E_1=\{x:s(x)=\orpr x1\}$.

If $x\in E_0$ then by continuity of~$s$ there is an interval $(p_x,q_x)$
that contains~$x$, with rational end points,
and such that $s[(p_x,q_x)]\subseteq [0,\orpr x0]$.
Here we use that $\Q\subseteq X$: we can talk without ambiguity about 
intervals with rational end points.
It is clear that when $x<y$ in~$E_0$ we have 
$y\in(p_y,q_y)\setminus(p_x,q_x)$, and it follows that $x\mapsto(p_x,q_x)$
is injective. We deduce that $E_0$~is countable.
Likewise one shows that $E_1$~is countable.

We see that $E(f)\cap A$ is countable, and because $A$~is a Bernstein set
it follows that $E(f)$~itself is countable

\smallskip
Next we let $x\in C(f)\cap A$ such that $f(x)\in V$.
Then $f(x)\notin Y$ and so $f(x)$~is split in~$\A_Y$,
and $s(x)\in \{\orpr{f(x)}0,\orpr{f(x)}1\bigr\}$; we split
$\{x\in C(f)\cap A:f(x)\in V\}$ 
into $C_0=\bigl\{x:s(x)=\orpr{f(x)}0\bigr\}$ 
and $C_1=\bigl\{x:s(x)=\orpr{f(x)}1\bigr\}$.

As above we take for $x\in C_0$ an interval $(p_x,q_x)$ that contains~$x$,
with rational end points, and such that 
$s[(p_x,q_x)]\subseteq [0,\orpr{f(x)}0]$.
If $x\neq y$ in~$C_0$ then $f(x)\neq f(y)$ because $f$~is injective on~$C(f)$;
if, say, $f(x)<f(y)$ then $y\in(p_y,q_y)\setminus(p_x,q_x)$ and it follows
that $x\mapsto(p_x,q_x)$ is injective.
We conclude, as above, that $C_0$~is countable, as is~$C_1$.

\smallskip
We see that $f[C(f)\cap A]\cap V$ is countable and hence, 
by the properties of the family $\{V\}\cup\{A_\beta:\beta\in\cee\}$
in Proposition~\ref{prop.family}, 
that $f\bigl[S(f)\bigr]$ does not have cardinality~$\cee$.
But as noted in the remarks before that proposition this means that
$f\bigl[S(f)\bigr]$~is countable. 

Thus we see that $f[G]=f\bigl[E(f)\bigr]\cup f\bigl[S(f)\bigr]$ is countable.
\end{proof}

Now let $X$ and $Y$ be subsets of~$\cee$ such that $X\notsubseteq Y$ and take
$\alpha\in X\setminus Y$.
Then $S_X$ and $S_Y$ satisfy the conditions of Lemma~\ref{lemma.range}.
Indeed, by definition we have $\Q\cup A_\alpha\subseteq S_X$ and
$(A_\alpha\cup V)\cap S_Y=\emptyset$.

The lemma then tells us that every continuous map $s:K_X\to K_Y$ has a 
countable range, so that $K_Y$~is not a continuous image of~$K_X$. 

\section{Murray Bell's answer}

In \cite{MR0735899} Murray Bell constructed a sequence~$\Nseq{B}$ of 
Boolean subalgebras of~$\Pwmodfin$ with the following properties:
\begin{itemize}
\item The algebra $B_1$ is ccc but not $\sigma$-$2$-linked, and
\item for $n\ge2$ the algebra $B_n$ is $\sigma$-$n$-linked
      but not $\sigma$-$(n+1)$-linked.
\end{itemize}
By definition a Boolean algebra~$B$ is $\sigma$-$n$-linked if one can write
it as a union~$\bigcup_{k\in\omega}L_k$ of countably many subsets that 
are $n$-linked, 
which means that $\bigwedge F>0$ whenever $F$~is an $n$-element subset 
of~$L_k$.

Since a subalgebra of a $\sigma$-$n$-linked algebra is 
again $\sigma$-$n$-linked it follows at once that $B_m$~cannot be embedded 
in~$B_n$ whenever $m<n$.
Unfortunately the constructions in~\cite{MR0735899} are such that the 
sequence~$\Nseq{B}$ is not decreasing, 
nor does there seem to be a straightforward 
way of embedding~$B_n$ into~$B_m$ when $m<n$.

Fortunately there is a relatively easy way out via Stone duality.
For each~$n\in\N$ let $S_n$ be the Stone space of~$B_n$, and then let
$X_n$ be the one-point compactification of the topological sum
$\bigoplus_{k\ge n}S_k$, with~$\infty_n$ its point at infinity.

\begin{lemma}
For $n\ge2$ the clopen algebra~$C_n$ of~$X_n$ is $\sigma$-$n$-linked.
\end{lemma}

\begin{proof}
For $k\ge n$ we have $B_k\setminus\{0\}=\bigcup_{m\in\omega}L_{k,m}$ 
where each~$L_{k,m}$ is $k$-linked and, a fortiori, $n$-linked.

For $k\ge n$ and $m\in\omega$ let
$D_{k,m}=\{C\in C_n\setminus\{0\}:C\cap S_k\in L_{k,m}\}$.
Each family $D_{k,m}$~is $n$-linked, hence $C_n$~is $\sigma$-$n$-linked.
\end{proof}

This ensures, as above, that $C_m$ cannot be embedded into~$C_n$ when $m<n$.

\bigskip
To show that $\Nseq{C}$ corresponds to a sequence of subalgebras
of~$\Pwmodfin$ we define onto mappings between the Stone spaces.

If $m<n$ then we can map $X_m$ onto~$X_n$, and hence embed $C_n$ into~$C_m$,
by 
\begin{itemize}
\item mapping $x$ to $\infty_n$ when 
      $x\in \{\infty_m\}\cup\bigcup_{m\le k<n}S_k$, 
\item mapping $x$ to itself if $x\in\bigcup_{k\ge n}S_k$.
\end{itemize}

It remains to map $\omega^*$ onto~$X_1$.
To this end look at the countable set~$\N\times\omega$ and divide it into the
vertical lines~$V_n=\{n\}\times\omega$.
For each~$k$ we have a continuous surjection $f_k:V_k^*\to S_k$.
The union~$f$ of the maps~$f_n$ is a continuous map 
from~$\bigcup_{k\in\N}V_k^*$ onto~$\bigoplus_{k\in\N}S_k$, and because
$(\N\times\omega)^*$ is an $F$-space we can extend $f$ to a continuous
map from $(\N\times\omega)^*$ onto~$X_1$.

The sequence $\Nseq{C}$ of Boolean algebras is as desired.

\subsection*{Acknowledgment} 
Thanks to Alan Dow for discussions leading
to the last section.

\end{document}